\documentclass[11pt]{article}
\usepackage{amsmath}
\usepackage{amssymb}
\usepackage{amsthm}
\usepackage{enumerate}
\usepackage{graphicx,color}
\usepackage{amsfonts,amsmath,amssymb}
\usepackage{url}
\usepackage{epsfig,latexsym,graphicx}
\usepackage{esint}
\usepackage{graphicx,latexsym,amssymb,amsmath,amsfonts}
\newtheorem{theorem}{Theorem}[section]
\newtheorem{lemma}[theorem]{Lemma}

\newtheorem{proposition}[theorem]{Proposition}

\theoremstyle{definition}
\newtheorem{remark}{Remark}[section]

\newcommand{\pr}[1]{\operatorname{\mathbf{P}}\left(#1\right)}

\newcommand{\E}[1]{\operatorname{\mathbf{E}}\left[#1\right]}

\newcommand{\vol}[1]{\operatorname{vol}\left(#1\right)}

\newcommand{\critical}{\mathrm{c}}
\newcommand{\critdet}{{det}}

\newcommand{\tdet}{T_{det}}

\newcommand{\IGNORE}[1]{}

\newcommand{\speed}{{S}}

\begin{document}

\title{Phase transition for finite-speed detection\\  among moving particles}

\author{Vladas Sidoravicius\thanks{Instituto de Matem\'atica Pura e Aplicada, Rio de Janeiro, Brazil.}
\and Alexandre Stauffer\thanks{Department of Mathematical Sciences, University of Bath, U.K. Supported in part by a Marie Curie Career Integration Grant PCIG13-GA-2013-618588 DSRELIS.}}
\date{\today}
\maketitle

\begin{abstract}
Consider the model where particles are initially distributed on
$\mathbb{Z}^d, \, d\geq 2$, according to a Poisson point process of intensity
$\lambda>0$,  and are moving in continuous time as independent simple
symmetric random walks.
We study the escape versus detection problem, in which the target, initially
placed at the origin of $\mathbb{Z}^d, \, d\geq 2$, and changing its location on the lattice
in time according to some rule, is said to be detected if at some finite time
its position coincides with the position of a  particle.
We consider the case where the target can move with speed at most 1,
according to any continuous function and can adapt its motion based on
the location of the particles.
We show that there exists sufficiently small $\lambda_* > 0$, so that
if the initial density of  particles $\lambda < \lambda_*$, then the target can avoid detection forever.
\newline
\newline
\emph{Keywords and phrases.} Poisson point process, target detection, oriented space-time percolation.
\newline
MSC 2010 \emph{subject classifications.}
Primary 82C43; %Time-dependent percolation
Secondary 60G55, % Point processes
          60K35. % Interacting random processes; statistical mechanics type models; percolation theory
\end{abstract}

%\section*{Full Paper}
%\setcounter{footnote}{0}
%\setcounter{page}{1}
%############################################################################################
%############################################################################################
%############################################################################################
\section{Introduction}\label{sec:intro}

Let $\Pi$ be a Poisson point process of intensity $\lambda>0$ on $\mathbb{Z}^d, \, d\geq2 $.
We label all points of this process by positive integers in some arbitrary way, {\em{i.e.}}
$\Pi = \{ p_j \}_{j\geq 1}$, and interpret the points of $\Pi$ as \emph{particles}. We denote by $\eta_j (0), \; j \geq 1$, the initial
position of the $j^{\text{th}}$  particle, and we will assume that each particle $p_i, \, i \geq 1$, moves as an independent continuous-time random walk  on
$\mathbb{Z}^d$. More formally,  for each $k\geq 1$, let $(\zeta_k(t))_{t\geq 0}$ be an
independent  continuous-time random walk on $\mathbb{Z}^d$ starting from the origin.
Then  $\eta_k(t):=\eta_k(0) + \zeta_k(t)$ denotes the location of the $k$-th particle at time $t$.

In addition, we consider an extra particle, called \emph{target}, which at time $0$ is
positioned at the origin, and is moving  on $\mathbb{Z}^d, \, d \geq 2$ in time,  according to a certain
prescribed rule.
We say that the target is \emph{detected} at time $t$, if there exists a particle $p_j$ located
at time $t$  at the same vertex as the target.
We will assume that the target particle wants to evade detection and can do so
by moving in continuous time according to any continuous function on $\mathbb{Z}^d$,
which can depend on the past, present and future positions of the particles.

\noindent More precisely, let $\mathcal{P}$ be the set of functions $g\colon \mathbb{R}_+ \to \mathbb{Z}^d$ such
that:
\begin{align}
\label{j1}
&\text{for any $g \in \mathcal{P}$, any $t\geq 0$ and any $\xi>0$, if $\|g(t+\xi)-g(t)\| > 1$ then} \notag\\
&\text{there exists $\xi'\in(0,\xi)$ for which $\|g(t+\xi')-g(t)\| < \|g(t+\xi)-g(t)\|$.}
\end{align}
We view $\mathcal{P}$ as the set of all permitted trajectories for the target, and
 $g(t), \; g \in \mathcal{P}$, denotes the
position of the target at time $t$. The condition (\ref{j1})  in the definition of $\mathcal{P}$
prevents the target to make long range jumps, i.e. for any trajectory $g\in\mathcal{P}$, the target  is allowed to jump only  between nearest neighbor vertices of $\mathbb{Z}^d$.

We say that $g\in \mathcal{P}$ is \emph{detected} at time $t$ if there exists a particle
 $p_j \in \Pi$, for some $j\geq 1$, such that $\eta_j(t) = g(t)$, and define the detection time of $g$ as follows:
 $$
   \tdet(g) = \inf\Big\{t \geq 0 \colon g(t) \in \bigcup\nolimits_{k\geq 1} \eta_k(t)\Big\}.
$$
In~\cite[Theorem~1.1]{Stauffer11} it was shown that there exists a phase transition in $\lambda$ so
that, if $\lambda$ is large enough, for all $g\in\mathcal{P}$ we have $\tdet(g)<\infty$
almost surely.  Hence, the target cannot avoid detection forever even if it knew the
past, present and \emph{future} positions of the particles at all times, and could move at any time at any arbitrarily large speed.

Here we consider a parameter $0 < \speed < + \infty$ and let
$\mathcal{P}_\speed\subset \mathcal{P}$ be the set of all trajectories $g\in\mathcal{P}$
with maximum \emph{speed} $\speed$, {\it {i.e.}},
$$
\text{$\mathcal{P}_\speed := \{$ $g\in\mathcal{P}$:  $\forall t\geq 0$  $\forall \xi>0$,  $\|g(t+\xi)-g(t)\| \leq \xi \speed \lor 1 \}$.}
$$
Then define
$$
\lambda_\critdet(\speed)
= \inf\Big\{\lambda \geq 0 \colon \pr{\tdet(g) < \infty}
= 1 \text{ for all }g\in\mathcal{P}_\speed\Big\}
$$
and
$$
\lambda_\critdet(\infty)
= \inf\Big\{\lambda \geq 0 \colon \pr{\tdet(g) < \infty}
= 1 \text{ for all }g\in\mathcal{P}\Big\}.
$$
The main result in~\cite[Theorem~1.1]{Stauffer11}, mentioned above,  gives that $\lambda_\critdet(\infty)\in(0,\infty)$. Since
for any $\speed\leq \speed'$ we have
$\mathcal{P}_\speed \subseteq \mathcal{P}_{\speed'}$, then
$$
\lambda_\critdet(\speed) \leq \lambda_\critdet(\speed')\leq \lambda_\critdet(\infty) < \infty.
$$
It was also observed in ~\cite{Stauffer11}, that for sufficiently small $\lambda > 0$, there is 
a strictly positive probability for the target, starting from the origin,  to avoid detection forever, provided it can move 
at any time at any arbitrarily large speed, {\it {i.e.}} $\lambda_\critdet(\infty) > 0$.

The main contribution of this work is to establish an analogous result for any bounded speed, {\it i.e.} to show the existence of a non-trivial phase transition for all finite speeds $0 < \speed < + \infty$.
In other words, for any $\speed>0$, if the density $\lambda$ of particles  is 
small enough, with positive probability a target moving with maximum speed $\speed$ can
avoid detection forever.
\begin{theorem}
\label{thm:main}
   For any $\speed>0$, we have $\lambda_\critdet(\speed)>0$.
\end{theorem}

\begin{remark}
   In many of the references mentioned in the related work discussion below~\cite{Kesidis,Konst,PSSS11,PeresSousi11,Stauffer11},
   the problem of target detection was considered in a continuous-space variant of the model.
   In this variant, particles are given by a Poisson point process of intensity $\lambda$ on $\mathbb{R}^d$, and move independently as
   Brownian motions. Then, we say that the target is detected at time $t$ if there exists a particle within distance~$1$ from the target at that time.
   This variant is an extension of the widely studied \emph{Boolean model} (also called \emph{random geometric graph} or \emph{continuum percolation})
   to a mobile setting. We highlight that, with little change in the proof, Theorem~\ref{thm:main} can also be shown to hold in this continuous-space version.
   We discuss how to change our proof to this setting in Section~\ref{sec:brownian}.
\end{remark}

%\begin{remark}
%{\rm
%   Establish percolation of the vacant set of oriented random interlacements.
%%}
%\end{remark}

\medskip

\noindent {\bf Related work.} The problem of detecting a target by moving particles has been studied in
other settings. For example, ~\cite{Kesidis,Konst}
considered the continuous version of this model, where particles move as
Brownian motion in $\mathbb{R}^d$, and studied the case where the target
is non-mobile and stays put at the origin (using our notation, this corresponds to $g\equiv 0$).  Using arguments from
stochastic geometry, they derived the precise distribution of the
detection time; in particular, they showed that
\begin{equation}
   \pr{\tdet(g)>t} = \exp\left(-\lambda \vol{W_d(t)}\right) \quad\text{ when $g\equiv 0$},
   \label{eq:wiener}
\end{equation}
where $W_d(t)$ is the $d$-dimensional Wiener sausage up to time
$t$. The volume of the Wiener sausage is known to grow as $\sqrt{t}$ in $d=1$, $\frac{t}{\log t}$ in $d=2$, and $t$ in $d\geq 3$.

For the case of a mobile target, if the target has to move \emph{independently} of the particles (i.e., $g$ is a deterministic function),
in\ \cite{PSSS11} it was shown that, for any given $g$,
a similar expression as in~\eqref{eq:wiener} holds with $W_d(t)$ replaced by a Wiener sausage with drift $-g$.
Also,~\cite{PSSS11}, and in particular~\cite{PeresSousi11}, showed that, among all deterministic functions $g$,
the one that maximizes $\pr{\tdet(g)>t}$ is $g\equiv 0$. In other words, if the target has to move independently of the particles, the
best strategy for the target to avoid detection is to stay put.
See also the corresponding result for random walks on
$\mathbb{Z}^d$ in~\cite{DGRS}.
For the case where the motion of the target may depend on the positions
of the particles, it is shown in~\cite[Theorem~1.1]{Stauffer11} via a multi-scale analysis
that, for sufficiently large
$\lambda$, the target cannot avoid detection almost surely even if it knows beforehand the position of all particles at
all times.
A result of similar flavor was established in~\cite[Proposition~8]{KS05} for the study of the rate at which an infection spreads
among moving particles.
The result in~\cite[Theorem~1.1]{Stauffer11} gives in fact more information.
It establishes that, provided $\lambda$ is large enough, $\pr{\tdet(g)>t}$ decays at least as quickly as
$\exp\left(-\frac{C\sqrt{t}}{\log^c t}\right)$ in $d=1$,
$\exp\left(-\frac{C t}{\log^c t}\right)$ in $d=2$,
$\exp\left(-C t\right)$ in $d\geq 3$. This bound is tight (up to the constant factor $C$) and matches up with the case $g\equiv 0$ for $d\geq 3$.
Intuitively, this gives that a target that knows the positions of all nodes at all times cannot evade detection much longer than a non-mobile target.

%############################################################################################
%############################################################################################
%############################################################################################

\section{Proof of Theorem~\ref{thm:main}}

The hardest case is to prove Theorem~\ref{thm:main} in two dimensions. In higher dimensions, we 
simply show that the target can avoid detection by moving only in the first two dimensions; i.e., 
we define 
the hyperplane
\begin{equation}
  \mathbb{H}_d = \mathbb{Z}^2 \times O_{d-2},
  \label{eq:defh}
\end{equation}
where $O_{d-2}$ stands for the origin of $\mathbb{Z}^{d-2}$, and show that the target can avoid detection by only moving 
within $\mathbb{H}_d$. (In the case $d=2$, we simply define $\mathbb{H}_2=\mathbb{Z}_2$.)

For any $i\in \mathbb{H}_d$, consider the time interval
\begin{equation}
   T_i=\left[\frac{\|i\|_1}{\speed},\frac{\|i\|_1+1}{\speed}\right],
   \label{eq:ti}
\end{equation}
and the space-time line segment
$$
   K_{i} = i \times T_i.
$$
We will show that for $\lambda$  small enough, there exists a trajectory $g$ for the
target that is contained in the space-time region $\bigcup_{i\in\mathbb{H}_d}K_i$
and is never detected. Note that, for such a trajectory $g$, we have $g\in\mathcal{P}_\speed$.
We say that $K_i$ is $\emph{vacant}$ if there is no particle of $\Pi$ inside $K_i$, and
 $E_i$ will denote the indicator random variable that $K_i$ is vacant, i.e.
$E_i := \mathbb{I}_{\{\text{$K_i$ is vacant}\}}$.
We will show that for small enough $\lambda$, the process induced by $\{E_i\}_{i\in\mathbb{H}_d}$ stochastically
dominates an independent supercritical oriented percolation process
on the square lattice. 
%%%%%%%%%%%%%%%%%%%%%%%%%%%%%%%%%%%%%%%%%
%
%        For two processes $\{X_i\}_{i\in \mathbb{Z}^d}$
%        and $\{Y_i\}_{i\in \mathbb{Z}^d}$, we say that $X$ stochastically dominates
%        $Y$ if there exists a coupling between the probability measures of $X$ and
%        $Y$ so that $\pr{\bigcap\nolimits_{i\in\mathbb{Z}^d}\{X_i\geq Y_i\}}=1$.
%
%%%%%%%%%%%%%%%%%%%%%%%%%%%%%%%%%%%%%%%%%
\begin{proposition}
\label{pro:stochdom}
   For any $\lambda >0$ and $\speed>0$,  there exists
   $p=p(\lambda,\speed) >0$, so that if $\{X_i\}_{i\in\mathbb{H}_d}$
   are i.i.d.\ Bernoulli random variables taking values $0$ or $1$ with mean $p$, then
   $\{E_i\}_{i\in\mathbb{H}_d}$ stochastically dominates $\{X_i\}_{i\in\mathbb{H}_d}$.
   Moreover, for any $\speed>0$, we have
   \begin{equation}
   \label{eq:domin}
      \liminf_{\lambda\downarrow 0}p(\lambda,\speed)= 1.
   \end{equation}
\end{proposition}

\noindent The proof of Theorem~\ref{thm:main} is a straightforward application
   of Proposition~\ref{pro:stochdom}.
\begin{proof}[{\bf Proof of Theorem~\ref{thm:main}}]
   Eq. (\ref{eq:domin}) of Proposition \ref{pro:stochdom} implies that, given $S>0$,
   there exists $\lambda_c (S)>0$, such that for $0< \lambda < \lambda_c  (S)$, we have $p(\lambda,S)>p_\critical$, where
   $p_\critical$ is the critical probability for oriented site percolation on $\mathbb{Z}^2$.
   Hence, with positive probability, there exists an infinite oriented path of adjacent
   sites of $\mathbb{H}_d$, say $i_0 = 0, i_1,i_2,\ldots$, such that for all $j\geq 0$ we have $\|i_j\|_1=j$ and $K_j$ is
   vacant. Thus, the path $g\in\mathcal{P}_\speed$ which follows the segment $K_{i_j}$ in the time
   direction, and at time $\frac{\|i_j\|_1+1}{\speed}$ moves to
   $K_{i_{j+1}}$ and then follows along $K_{i_{j+1}}$ until the next jump to $K_{i_{j+2}}$, etc., for all $j \geq 0$, is the path
   for which $\tdet(g)=\infty$. 
\end{proof}

%############################################################################################
%############################################################################################
%############################################################################################

\section{Proof of Proposition~\ref{pro:stochdom}}

For any $k \geq 1$, let $J_k: = \{ x \in \mathbb{H}_d \, : \, ||x||_{1} = k \}$, and
 $\mathcal{G}_k$ be the $\sigma$-algebra
generated by $\{E_i\}_{i\in\mathbb{H}_d\colon\|i\|_{_1} \leq k}$. 
The goal of this section is to show that, for $k\geq 1$,
the following holds:
\begin{equation}
   \text{conditioned on any $G\in \mathcal{G}_{k-1}$, } \{E_i\}_{i\in\mathbb{H}_d\colon i\in J_{k}}
   \text{ stochastically dominates } \{X_i\}_{i\in\mathbb{H}_d\colon i\in J_{k}}.
   \label{eq:level}
\end{equation}

\noindent We will analyze the states of sites of $J_k$ inductively on $k=0,1,\ldots$
Once~\eqref{eq:level} is established, Proposition~\ref{pro:stochdom} follows directly.
The proof of~\eqref{eq:level} will be split in several steps and lemmas.
We start with an informal  description of the proof, discussing the main
ingredients used to establish~\eqref{eq:level}, and then proceed to the rigorous arguments.
%############################################################################################
%\subsection*{High-level overview}

The main idea of the proof is the following: by definition, the space-time region
$\bigcup_{i\in\mathbb{H}_d}K_i$ grows linearly in time and moves away from the origin
at linear speed. In particular, for any time $t$,
the site $i$, such that $t\in T_i$, has $\ell_1$ norm of order $t$.
Since by time $t$ a particle, performing simple symmetric random walk, typically moves a distance
of order $\sqrt{t}$, it implies that each individual particle
can spend only a limited amount of time
inside the region $\bigcup_{i\in\mathbb{H}_d}K_i$. Thus, if the intensity of the
Poisson point process is sufficiently small, we will show that the union of all vacant $K_i$'s contains an infinite connected
component; i.e., the region of
$\bigcup_{i\in\mathbb{H}_d}K_i$ that is not visited by particle ``percolates'' in space-time.

To make the above argument rigorous,  fix $\lambda >0$,  small enough, such that there exists $1 \leq k_0 < + \infty$, so that, with
sufficiently large probability, there is no particle in the space-time region $\bigcup_{i\in J_k}K_i$ for all $k\leq k_0$.
Let $k=k_0+1$, and select all particles that visit the space-time region $\bigcup_{i\in J_k}K_i$. Let $u$ be one such particle.
We observe the motion of $u$ from the time it first visits $\bigcup_{i\in J_k}K_i$ onwards.
In order to do this, we introduce the \emph{region of influence} of $u$, 
which is a random region given by a ball centered at the space point which is the canonical space-coordinate projection of the space-time point where $u$
first visits $\bigcup_{i\in J_k}K_i$, and which has a random radius that depends on the motion of $u$ from that time onwards.
This region of influence will intersect all sites $i'$ of $\mathbb{H}_d$
for which $u$ can enter $K_{i'}$. As discussed above, $u$ can only spend a finite time inside
$\{K_i\}_{i\in \mathbb{H}_d}$, so the region of influence of $u$ is bounded. We show that the region of influence of $u$
has a radius with an exponentially decaying tail.

For a general level $k$, we repeat the argument above:  among all particles that
enter the space-time region $\bigcup_{i\in J_{k}}K_i$ select only those  which have not entered the space-time region
$\bigcup_{j=0}^{k-1}\bigcup_{i\in J_{j}}K_i$, and then define their region of influence in a similar way.
The goal is to show that the sites of $\mathbb{H}_d$ that do \emph{not} belong to the region of influence of any particle stochastically
dominates an independent percolation process that is known to be supercritical.

\bigskip

%############################################################################################
%\subsection*{Rigorous proof}
Now we begin the rigorous proof of Proposition~\ref{pro:stochdom}.
First we establish~\eqref{eq:level}. For $k=0$ the set $J_k$ has only one element and~\eqref{eq:level} holds
in a trivial manner. Now fix $k\geq1$ and let $\Psi_0=\Pi$.
Consider the particles that did not enter the space-time region
$\bigcup_{j=0}^{k-1}\bigcup_{i\in J_j}K_i$, and let $\Psi_k$ be the point
process determined by the location of these particles at time $\frac{k}{S}$.
\begin{lemma}\label{lem:ppp}
   For any $k \geq 0$,  $\Psi_k$ is a non-homogenenous Poisson
   point process of intensity uniformly bounded above by $\lambda$.
\end{lemma}
\begin{proof}
   Let $\Upsilon$ be the point process determined by the location of the
   particles of $\Psi_0$ at time $k/\speed$,
   %Clearly, $\Upsilon$
   which is a Poisson point
   process of intensity $\lambda$. For any $x$, let $p(x)$ be the probability
   that a random walk that at time $k/\speed$ is located at $x$ does not visit
   $\bigcup_{j=0}^{k-1}\bigcup_{i\in J_j}K_i$ during $[0,k/\speed)$.
   Then, $\Psi_k$ is a Poisson point process obtained by \emph{thinning}
   $\Upsilon$ in such a way that its intensity measure at position $x$ is
   $\lambda p(x) \leq \lambda$.
\end{proof}

\noindent For each $i\in J_k$, let
\begin{align*}
N_i :=  \;
& \text{number of particles of $\Psi_k$ that visit the set $J_k$ during} \\
& \text{the interval $[k/\speed,(k+1)/\speed]$ and enter $J_k$ through $i$.}
\end{align*}

\begin{lemma}\label{lem:numberhit}
   There exists a positive constant $c=c(d,\speed)$ so that
   the set $\{N_i\}_{i\in \mathbb{H}_d}$ is stochastically dominated by $\{M_i\}_{i\in \mathbb{H}_d}$, where $M_i$ are i.i.d.\ Poisson random variables of
   mean $c\lambda$.
\end{lemma}
\begin{proof}
   We define a set of random variables $\{N_i'\}_{i\in J_k}$ which are distributed independently across different values of $k$.
   For any given $k$,
   consider an independent configuration of particles distributed as a Poisson point process of intensity $\lambda$ over $\mathbb{Z}^d$.
   Let each particle perform a continuous-time random walk for time $1/\speed$. Then, for each $i\in J_k$, let
   $N_i'$ be the number of particles that visit $i$ during $(0,1/\speed)$ and visit $i$ before visiting any other site of $J_k$.
   By Lemma~\ref{lem:ppp} and independence across different values of $k$, we have that
   $\{N_i'\}_{i\in \mathbb{H}_d}$ stochastically dominates $\{N_i\}_{i\in \mathbb{H}_d}$.
   It then suffices to show that, for any given $k$,
   $\{N_i'\}_{i\in J_k}$ is stochastically dominated by $\{M_i\}_{i\in J_k}$.

   By thinning of Poisson point processes we have that $\{N_i'\}_{i\in J_k}$ are independent Poisson random variables.
   It remains to show that there exists a constant $c=c(d,\speed)$ so that, uniformly for all $i$, we have $\E{N_i'}\leq c\lambda$.
   Fix $i\in J_k$ and let $\tilde p(x)$ be the probability that a particle starting from $x\in\mathbb{Z}^d$ visits $i$ during $[0,1/\speed)$ and does so before
   visiting any other site of $J_k$. Then, we have that
   $$
      \E{N_i'} = \lambda \sum_{x\in\mathbb{Z}^d}\tilde p(x).
   $$
   Since the number of jumps
   of a particle during $[0,1/\speed)$ is a Poisson random variable of mean $1/\speed$,
   there is a constant $c_1>0$ such that, for any $x$ so that $\|x-i\|_1\geq 2/\speed$, we have $\tilde p(x) \leq e^{-c_1\|x-i\|_1}$. Then, using that the
   number of sites at distance $z$ from $i$ is at most $c_2 z^{d-1}$ for some constant $c_2>0$, we have
   $$
      \E{N_i'} \leq  \lambda \left( \sum_{x\colon \|x-i\|_1< 2/\speed}\tilde p(x) + \sum_{z=2/\speed}^\infty c_2 z^{d-1}e^{-c_1 z}\right)
      \leq c\lambda,
   $$
   for $c=c(d,\speed)$ sufficiently large.
\end{proof}

We now introduce some notations that we will use to define the region of influence of a site.
Fix $\delta=\frac{\speed}{4\sqrt{d}}$ and let $C^\delta_{0,0} \equiv C^\delta \subset \mathbb{Z}^d \times \mathbb{R}_+$ be the space-time cone
$$
   C^\delta= \{(y,t) \colon y\in\mathbb{Z}^d, t\geq 0 \text{ and }\|y\|_2 < \delta t\}.
$$
We claim that for any $x\in \mathbb{H}_d$ and any $t\in T_x$, the shifted cone $C^\delta_{x,t}=(x,t)+C^\delta$ does not intersect $K_j$ for any $j\neq x$.
In order to see this, let $j\in\mathbb{H}_d$ be such that $\|j\|_1\geq \|x\|_1$. Then, for any $s$ for which $(j,s)\in K_j$ we have
\begin{align*}
   s-t
   \leq \frac{1}{\speed}+\frac{\|j\|_1-\|x\|_1}{\speed}
   & \leq \frac{1+\|j-x\|_1}{\speed} \\
   & \leq \frac{1+\sqrt{d}\|j-x\|_2}{\speed}
   \leq \frac{1+\|j-x\|_2}{4\delta}
   \leq \frac{\|j-x\|_2}{2\delta}.
\end{align*}
On the other hand, by the definition of $C^\delta$, for any $(j,s')\in C^\delta_{x,t}$ we have $s'-t > \frac{\|j-x\|_2}{\delta}$.

For a random walk $(\xi(t))_t$ that starts from the origin define $\tau$ as
the last time that $(\xi(t))_t$ is outside $C^\delta$; i.e.,
$$
   \tau = \inf\{t \geq 0 \colon (\xi(s),s) \in C^\delta \text{ for all } s\geq t\}.
$$
Now define the random variable
\begin{equation}
   \chi = \sup\{\|\xi(t)\|_2 \colon t\in [0,\tau]\}.
   \label{eq:chi}
\end{equation}
The definition of $\chi$ is illustrated in Figure~\ref{fig:mushroom}(a).
\begin{figure}[tbp]
   \begin{center}
      \includegraphics[scale=.8]{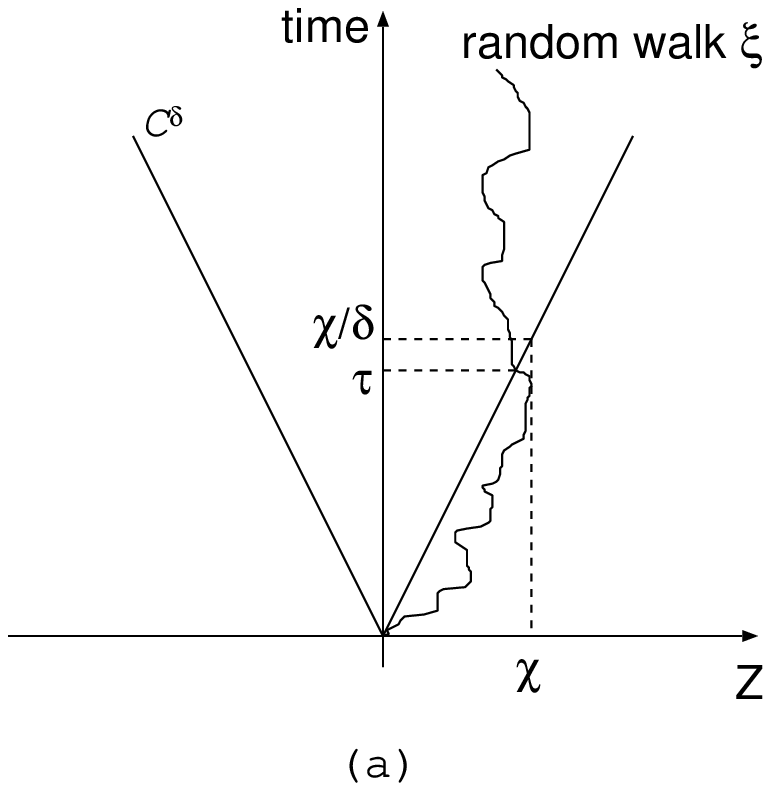}
      \hspace{\stretch{1}}
      \includegraphics[scale=.8]{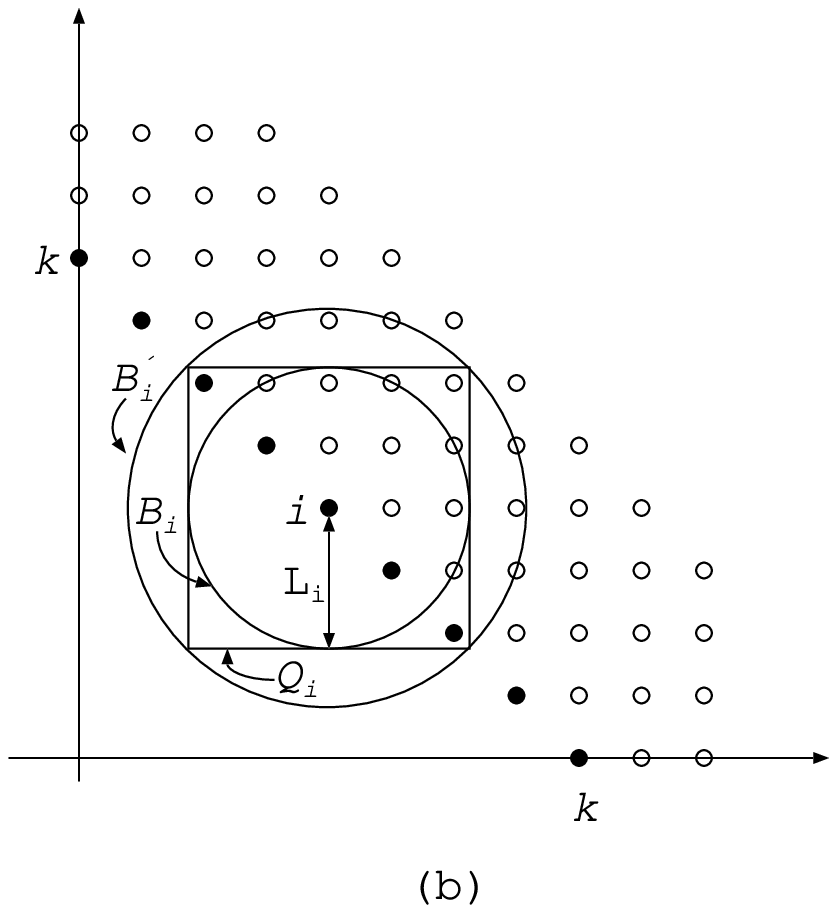}
   \end{center}\vspace{-.5cm}
   \caption{(a) Illustration of the space-time cone $C^\delta$ for $d=1$ and the definition of $\tau$ and $\chi$.
      (b) Illustration of the region of influence of site $i$ for $d=2$.}
   \label{fig:mushroom}
\end{figure}

We are now ready to define the region of influence of a site.
From now on we fix $k$ and $i\in J_k$, and
we denote by $B_i$ the region of influence of site $i$.
We couple $M_i$ and $N_i$ so that $M_i\geq N_i$.
If $M_i=0$, we set $B_i=\emptyset$. Otherwise we proceed as follows.
We construct a region for each of the $N_i$ particles that visit $K_i$.
Consider the $j$th such particle and let $\chi_j$ be an independent random variable distributed as $\chi$, and define
$t_j$ as the first time the particle visits $K_i$.
With this, define the space-time cylinder 
$$
   S_j=(B(i,\chi_j) \cap \mathbb{H}_d) \times [t_j,t_j+\chi_j/\delta],
$$
where $B(x,r)\subset\mathbb{Z}^d$ stands for the ball of radius $r$
centered at $x$.
Note that, for any time $s\geq t_j+\chi_j/\delta$, the
particle is inside the space-time cone $C_{i,t_j}^\delta$.
Consequently, at any time $s\geq t_j+\chi_j/\delta$, the
$j$th particle cannot intersect $\bigcup_{z\in\mathbb{H}_d} K_z$; hence the sites $\iota\in\mathbb{H}_d$ for which $j$ can intersect $K_\iota$ are contained in $S_j$.
Define $\chi_j$ in the same way as above for all $N_i < j \leq M_i$, and take $L_i=\max_{j=1}^{M_i} \chi_j$ and
$B_i = B(i,L_i) \cap \mathbb{H}_d$.
Note that $B_i$ contains all sites that intersect $\bigcup_{j=1}^{N_i}S_j$.
Since the $\{M_i\}_{i\in J_k}$ are i.i.d.\ random variables, the regions $\{B_i\}_{i\in J_k}$ are also i.i.d.

We have the following lemma bounding the size of $B_i$.
\begin{lemma}\label{lem:influence}
   There exist constants $c,c'>0$ independent of $\lambda$ such that, for all $x\geq 1$ and $i\in\mathbb{H}_d$,
   $$
%       \pr{\chi > x} \leq c\lambda \exp(-c'x)
%       \quad\text{and}\quad
      \pr{L_i > x} \leq c\lambda \exp(-c'x).
   $$
\end{lemma}
\begin{proof}
   First we derive an upper bound for $\pr{\chi \geq x}$. The probability that a random walk performs at least $x$ jumps in a time interval of length
   $x/2$ is $e^{-c_1 x}$ for some positive constant $c_1$.
   If this does not happen, then $\chi$ can only be at least $x$ if at some time after $x/2$ the random walk
   is outside the cone $C^\delta$. For any integer $a\geq 0$, let $I_a$ be the time interval
   $[z/2 + a, z/2 + a+1]$. We show that, during $I_a$, the probability that the distance between the random walk and the origin exceeds
   $\delta(x/2+a)$ is at most $e^{-c_2 (x+a)}$ for some positive constant $c_2=c_2(d,\speed)$.
   This follows since, with probability $1-e^{-c_3 (x+a)}$, the random walk is within distance
   $\frac{\delta(x/2+a)}{2}$ from the origin at time $x/2+a$ and,
   with probability $1-e^{-c_4(x+a)}$, the random walk performs less than $\frac{\delta(x/2+a)}{2}$ jumps during a time interval of length
   $1$.
   Then, summing over $a$ we obtain
   $$
      \pr{\chi \geq z } \leq c_5 e^{-c_6 z} \quad\text{for some positive constants $c_5=c_5(d,\speed)$ and $c_6=c_6(d,\speed)$.}
   $$
   From this, we obtain
   $$
      \pr{L_i > x}
      \leq \E{M_i}\pr{\chi > x}
      \leq c c_5 \lambda e^{-c_6 x},
   $$
   where $c$ comes from Lemma~\ref{lem:numberhit}.
\end{proof}

Now we refer to Figure~\ref{fig:mushroom}(b).
If $B_i=\emptyset$, set $Q_i=\emptyset$. Otherwise,
let $Q_i$ be the square $i+[-L_i,L_i]^d\cap \mathbb{H}_d$ of side length
$2L_i$; note that $B_i$ is inscribed inside $Q_i$.
Consider the $2$-dimensional circle $B_i'$ that circumscribe $Q_i$;
the radius of $B_i'$ is $\sqrt{2}L_i$.
Now consider any site $\iota\in B_i$
so that $\iota \in J_{k'}$ for some $k'\geq k$, and take
any oriented path from the origin to $\iota$. By construction, this path
must contain a site in $Q_i \cap J_k$.

Now, for any $i\in\mathbb{H}_d$, we define $Y_i=0$ if there exists a $j \in \mathbb{H}_d$ with $\|j\|_1=\|i\|_1$ for which
$i \in Q_j$. Otherwise, we set $Y_i=1$. From the argument above we have that we can couple $Y_i$ and $E_i$ so that $Y_i \leq E_i$.
Therefore, if $\{Y_i\}_{i\in J_k}$ stochastically dominates $\{X_i\}_{i\in J_k}$
we establish~\eqref{eq:level}. This last statement holds since the radius of $B_i'$ has an exponential tail by Lemma~\ref{lem:influence}. Also,
the sites $i\in\mathbb{H}_d$ for which $\|i\|_1=k$ form a one-dimensional line segment, thus we can apply a
result by Holroyd and Martin~\cite[Theorem~3]{HM12}, which establishes that $\{Y_i\}_{i\in J_k}$ stochastically dominates $\{X_i\}_{i\in J_k}$, where
$\{X_i\}_{i\in J_k}$ are i.i.d.\ Bernoulli random variables with mean approaching $1$ as $\lambda\to0$. This establishes~\eqref{eq:level} and completes the proof
of Proposition~\ref{pro:stochdom}.
%############################################################################################
%############################################################################################
%############################################################################################
\section{Brownian motions on $\mathbb{R}^d$\label{sec:brownian}}

In this section we discuss how the proof of Theorem~\ref{thm:main} can be adapted to the setting where
particles perform independent Brownian motions on $\mathbb{R}^d$, $d\geq 2$, and the target is
detected as soon as it is within distance $1$ from any particle.

The main changes needed in the proof regards the definition of the space-time region $K_i$ and the definition of the region of influence $B_i$.
We start with $K_i$. For all $i\in\mathbb{H}_d$, define $K_i=B(i,4/3) \times T_i$, where
$B(i,r)$ is the $d$-dimensional closed ball on $\mathbb{R}^d$ of radius~$r$ centered at~$i$, and $T_i$ is defined as in~\eqref{eq:ti}.
Then, the proof of Theorem~\ref{thm:main} (assuming Proposition~\ref{pro:stochdom}) carries through with no further changes, and it remains to
show how the proof of Proposition~\ref{pro:stochdom} needs to be changed to this setting.

The proof of Proposition~\ref{pro:stochdom} is composed of three lemmas. Lemma~\ref{lem:ppp} holds without any changes. For Lemma~\ref{lem:numberhit},
the only change we need is to define $N_i$ as the number of particles of $\Psi_k$ that visit $B(i,4/3)$ during the interval 
$[k/\speed,(k+1)/\speed]$, and first visit $B(i,4/3)$ not after 
visiting $B(j,4/3)$ for every $j\in J_k\setminus\{i\}$. (Note that we allow that the particle visits $B(i,4/3)$ concurrently to visiting $B(j,4/3)$ for
some $j\in J_k\setminus\{i\}$; in this case, this particle
counts to $N_i$ and an independent copy of the particle counts to $N_j$.)
Then Lemma~\ref{lem:numberhit} follows in the same way. 

For Lemma~\ref{lem:influence}, we need to do more changes since we need to define $B_i$ and $L_i$ differently.
From now on, fix $k$ and $i\in J_k$. 
Then let $x\in B(i,4/3)$ and $t\in T_i$ be arbitrary. 
We regard $x$ as the location and $t$ the time that the particle first visits $B(i,4/3)$.
Consider the cone $C^\delta_{x,t}=(x,t)+C^\delta$.
Then, for any $j\not\in B(i,5)$ and $s\in T_j$ we have
$$
   s-t \leq \frac{1+\|j\|_1-\|i\|_1}{\speed} \leq \frac{1+\|j-i\|_2}{4\delta} \leq \frac{1+4/3+\|j-x\|_2}{4\delta} \leq \frac{\|j-x\|_2}{2\delta},
$$
where in the second to last step we apply the triangle inequality, and in the last step we used that $\|j-x\|_2\geq 3$ since $j\not\in B(i,5)$.
Since, for any $(j,s')\in C^\delta$, it holds that $s'-t>\frac{\|j-x\|_2}{\delta}$, 
we obtain that $C^\delta$ does not intersect any $T_j$ for which $j\not\in B(i,5)$.
Now let $\ell$ be a particle from the set of the $N_i$ particles that visit $B(i,4/3)$ during the interval $[k/\speed,(k+1)/\speed]$,
and do so before visiting $B(i',4/3)$ for every $i'\in J_k\setminus\{i\}$.
Let $\chi_\ell$ be a random variable distributed as $\chi$ (cf~\eqref{eq:chi}), and let $L_i$ be the maximum of $\chi_\ell$ over all $\ell$.
Then, we set $B_i=B(i,10+L_i)$ if $M_i\geq1$. With these definitions, Lemma~\ref{lem:influence} holds without further changes and we obtain that the random variable $L_i$
has an exponential tail. Then, the remaining of the proof of Proposition~\ref{pro:stochdom} hold by setting 
$Q_i=i+[-10-L_i,10+L_i]^d \cap \mathbb{H}_d$ and $B_i'$ as the ball that circumscribe $Q_i$.
No further change is needed.

\bibliographystyle{plain}
\bibliography{detection}
\end{document}